\documentclass[11pt, a4paper]{amsart}
\usepackage{amsmath}
\usepackage{amssymb}
\usepackage{color}
\usepackage[utf8]{inputenc}

\newtheorem{prop}{Proposition}[section]
\newtheorem{teo}{Theorem}[section]
\newtheorem{lema}{Lemma}[section]
\newtheorem{coro}{Corollary}[section]

\theoremstyle{definition}

\def\ep{\varepsilon}

\def\a{\mathfrak q}
\def\R{\mathbb R}

\def\X{{\mathcal X}}
\def\H{{\mathcal H}}
\def\l{\lambda}
\def\ul{u^\l}

\def\l{\lambda}

\begin{document}
\title[Two-dimensional nonlocal diffusion
in exterior domains]{Asymptotic behavior for a nonlocal diffusion equation in exterior domains: the critical two-dimensional case}

\author[Cort\'{a}zar, Elgueta, Quir\'{o}s \and Wolanski]{C. Cort\'{a}zar, M. Elgueta, F. Quir\'{o}s \and N. Wolanski}

\address{Carmen Cort\'{a}zar\hfill\break\indent
Departamento  de Matem\'{a}tica, Pontificia Universidad Cat\'{o}lica
de Chile \hfill\break\indent Santiago, Chile.} \email{{\tt
ccortaza@mat.puc.cl} }

\address{Manuel Elgueta\hfill\break\indent
Departamento  de Matem\'{a}tica, Pontificia Universidad Cat\'{o}lica
de Chile \hfill\break\indent Santiago, Chile.} \email{{\tt
melgueta@mat.puc.cl} }

\address{Fernando Quir\'{o}s\hfill\break\indent
Departamento  de Matem\'{a}ticas, Universidad Aut\'{o}noma de Madrid
\hfill\break\indent 28049-Madrid, Spain.} \email{{\tt
fernando.quiros@uam.es} }

\address{Noem\'{\i} Wolanski \hfill\break\indent
Departamento  de Matem\'{a}tica, FCEyN,  UBA,
\hfill\break \indent and
IMAS, CONICET, \hfill\break\indent Ciudad Universitaria, Pab. I,\hfill\break\indent
(1428) Buenos Aires, Argentina.} \email{{\tt wolanski@dm.uba.ar} }

\thanks{All authors supported by  FONDECYT grants 1110074 and 1150028. The third author supported by
the Spanish Project MTM2011-24696. The fourth author supported by
 CONICET PIP625,
Res. 960/12, ANPCyT PICT-2012-0153, UBACYT X117, and MathAmSud 13MATH03. F. Quir\'{o}s and N. Wolanski would like to thank the Isaac Newton Institute for Mathematical Sciences, Cambridge, where part of this work was done during the
program \emph{Free Boundary Problems and Related Topics}.}



\date{}

\begin{abstract}
%\color{red}
We study the long time behavior of bounded, integrable solutions to a
nonlocal diffusion equation,  $\partial _t u=J*u-u$, where $J$ is a smooth, radially symmetric kernel with support $B_d(0)\subset\R^2$. The problem is set in an exterior  two-dimensional domain which excludes a hole $\H$, and with zero Dirichlet data on $\H$. In the far field scale, $\xi_1\le |x|t^{-1/2}\le \xi_2$ with $\xi_1,\xi_2>0$, the scaled function $\log t\, u(x,t)$ behaves as a multiple of the fundamental solution for the local heat equation with a certain diffusivity determined by $J$. The proportionality constant, which characterizes the first non-trivial term in the asymptotic behavior of the mass,  is given by means of  the asymptotic \lq logarithmic momentum' of the solution, $\lim_{t\to\infty}\int_{\R^2}u(x,t)\log|x|\,dx$. This asymptotic quantity can be easily computed in terms of the initial data. In the near field scale, $|x|\le t^{1/2}h(t)$ with $\lim_{t\to\infty} h(t)=0$,
the scaled function $t(\log t)^2u(x,t)/\log |x|$ converges to a multiple of $\phi(x)/\log |x|$, where $\phi$ is the unique stationary solution of the problem that behaves as $\log|x|$ when $|x|\to\infty$. The proportionality constant is obtained through a matching procedure with the far field limit.  Finally, in the very far field, $|x|\ge t^{1/2} g(t)$ with $g(t)\to\infty$, the solution is proved to be of order $o((t\log t)^{-1})$.
\end{abstract}

\maketitle

\date{}

\section{Introduction}
\label{Intro} \setcounter{equation}{0}

Let  $\mathcal{H}\subset \mathbb{R}^N$ be a non-empty bounded open set, which is assumed, without loss of generality, to satisfy
\begin{equation}
\label{hypotheses.H}
\tag{$H_\H$}
B_2(0)\subset \H\subset B_{\mathcal{R}}(0), \qquad \mathcal{R}\in (2,\infty).
\end{equation}
We do not assume $\H$ to be connected, so it may represent one or several holes in an otherwise homogeneous medium.
Our goal is to describe the long time behavior of solutions to the nonlocal diffusion problem
\begin{equation}\label{problem}
\tag{P}
\begin{cases}
\displaystyle\partial_t u(x,t)=Lu(x,t)\ &\mbox{in } (\R^N\setminus\H)\times\R_+,\\
u(x,t)=0&\mbox{in }\mathcal{H}\times\R_+,\\
u(x,0)=u_0(x)&\mbox{in }\mathbb{R}^N,
\end{cases}
\end{equation}
in the critical two-dimensional case $N=2$, thus completing the study for other dimensions performed by the authors in~\cite{CEQW,CEQW2,CEQW3}. The nonlocal operator $L$ is defined by $Lg:=J*g-g$, where the kernel $J$ is assumed to satisfy
\begin{equation}
\label{hypotheses.J}
\tag{$H_J$}
J\in C^2_{\textrm{c}}(\mathbb{R}^N)\quad \text{radially symmetric},\quad J>0\text{ if }|x|<d,\quad J=0\text{ if } |x|\ge d,  \quad \int_{\R^N} J=1.
\end{equation}
Diffusion models of this kind have been widely used to model the
dispersal of a species taking into account long-range effects
\cite{BZ, CF, Fi}, and also to describe phase transitions \cite{BCh1, BCh2, BChQ} and image enhancement \cite{GO}.

If the initial data $u_0$ are integrable and identically zero in the hole $\H$, problem~\eqref{problem} has a unique solution $u\in C\big([0,\infty);L^1(\R^N)\big)$. This is easily proved using Banach's fixed point theorem; see~\cite{CEQW}. However, in order to prove our asymptotic results for $N=2$, we will need further hypotheses on $u_0$. In the sequel we assume
\begin{equation}
\label{eq:hypotheses.u0.1}
\tag{$H_0$}
u_0\ge 0, \quad u_0=0\text{ in }\H,\quad u_0\in L^\infty(\mathbb{R}^2)\cap L^1(\R^2,\log|x|\,dx),
\end{equation}
plus some extra control of the growth of $u_0$ at infinity,
\begin{equation}
\label{eq:cond.second.momentum}
\tag{$H_1$}
u_0\in L^1(\R^2,|x|^2\,dx).
\end{equation}
An easy modification of the existence proof shows that if~\eqref{eq:hypotheses.u0.1} holds, then
\begin{equation}
\label{eq:log.momentum.L.infty}
u\ge0,\qquad\int_{\mathbb{R}^2} u(x,t)\log|x|\,dx<\infty\quad\text{for every }t>0,\qquad \|u\|_{L^\infty(\R^2\times\R_+)}\le \|u_0\|_{L^\infty(\R^2)}.
\end{equation}
Moreover, if $u_0$ satisfies in addition~\eqref{eq:cond.second.momentum}, then $u(\cdot,t)\in L^1(\R^2,|x|^2\,dx)$ for every $t>0$.

\

\noindent\textsc{Comparison with the case without holes. }
In the absence of holes, $\mathcal{H}=\emptyset$, the  mass $M(t)=\int_{\mathbb{R}^N}u(x,t)\,dx$ is conserved. If the initial data are not only integrable, but also bounded, the solution to~\eqref{problem} behaves for large times as the solution, $v$, to the local heat equation
\begin{equation}
\label{eq:definition.alpha}
\partial_t v=\a\Delta v\quad\text{in }\mathbb{R}^N\times\mathbb{R}_+,\qquad \a:=\frac1{2N}\int_{\R^N}|z|^2J(z)\,dz,
\end{equation}
with initial $v(\cdot,0)=u_0$; see~\cite{CCR, IR1}. More precisely,
\begin{equation*}
\label{eq:asymptotic.CCR}
\lim_{t\to\infty}t^{N/2}\max_{x\in\mathbb{R}^N}|u(x,t)-v(x,t)|=0.
\end{equation*}
Hence, the asymptotic behavior of $u$ can be described in terms of
the fundamental (self-similar) solution of~\eqref{eq:definition.alpha},
\begin{equation}
\label{eq:fundamental.solution}
\Gamma_{\a}(x,t)=t^{-N/2}U_\a\left(\frac{x}{t^{1/2}}\right),\qquad
U_\a(y)=(4\pi \a)^{-N/2}e^{-\frac{|y|^2}{4\a}}.
\end{equation}
Indeed, in self-similar variables we have convergence towards the
stationary profile  $MU_\a$, where
$M=\int_{\mathbb{R}^N}u_0$,
\begin{equation*}
\label{eq:behavior.whole.space}
\lim_{t\to\infty}\max_{y\in\mathbb{R}^N}|t^{N/2}u(yt^{1/2},t)-M U_\a(y)|=0.
\end{equation*}
Thus, there is an asymptotic symmetrization: no matter whether the
initial datum is radial or not,   the large time behavior of $u$ is
given by a radial profile, which, of course, has the same mass as
the datum.

The presence of holes introduces a technical difficulty, since Fourier transforms, which were the main tool in \cite{CCR, IR1}, can no longer be used. But there are differences of a more fundamental nature.
On the one hand, mass is not conserved. On the other hand, the presence of the
hole breaks (in general) the symmetry of the spatial domain, and an
asymptotic symmetrization is no longer possible.
It turns out that the asymptotic behavior depends strongly on the spatial dimension. It is already known that there is a big difference between the case of high dimensions, $N\ge3$, considered in~\cite{CEQW}, and the one-dimensional case studied in \cite{CEQW2, CEQW3}. This paper is devoted to  the intermediate, critical two-dimensional case. An analogous study for the local heat equation in dimensions $N\ge2$ was performed in~\cite{H}.

\

\noindent\textsc{Stationary solutions. } A main difference between the various dimensions has to do with stationary solutions of the problem, that is, functions $\phi$ satisfying
\begin{equation}\label{eq:stationary}
L\phi=0\quad \mbox{in }\R^N\setminus\H,\qquad
\phi=0\quad\mbox{in }\mathcal{H}.
\end{equation}
When  $N\ge 3$, there is a unique solution of this kind approaching  the constant 1 at infinity. Such a solution does not exist for low dimensions, $N=1,2$. Nevertheless, for $N=1$ there are stationary solutions that behave linearly at infinity. More precisely, given constants $b_\pm\ge0$, there is a unique solution to~\eqref{eq:stationary} satisfying
\begin{equation}
\label{eq:stationary.behavior.dim.1}
(\phi(x)-\max\{b^+x,-b^-x\})\in L^\infty(\mathbb{R}).
\end{equation}
In the critical two dimensional case there is  a stationary solution with a logarithmic behavior at infinity,
\begin{equation}
\label{eq:stationary.behavior}
\big(\phi(x)-\log|x|\big)\in L^\infty(\R^2\setminus\H).
\end{equation}
The construction of such solution, which is unique, is quite involved, and will be the subject of Section~\ref{sect:stationary.problem}.

\

\noindent\textsc{Conservation laws and  non-trivial asymptotic quantities. } Stationary solutions play an important role in our analysis. Indeed, though mass is not conserved, under adequate conditions on the initial data, there is in all dimensions a conservation law of the form
$$
\int_{\mathbb{R}^N} u(x,t)\phi(x)\,dx=\text{constant}
$$
for functions $\phi$ satisfying \eqref{eq:stationary} and having for each dimension the right behavior at infinity specified above. In the particular case $N=2$ we need $u_0\in L^1(\R^2,\log|x|\,dx)$.

If the initial data are also bounded, this conservation law yields on the one hand the large time behavior for the mass,
$$
\begin{array}{l}
M(t)\to \displaystyle M_\phi^*:=\int_{\mathbb{R}^N} u_0\phi>0\quad\text{if }N\ge3,\\[10pt]
M(t)=\begin{cases}
O(t^{-1/2})&\text{if }N=1,\\[10pt]
O((\log t)^{-1})&\text{if } N=2,
\end{cases}
\\[10pt]
\end{array}
$$
and on the other hand a global size estimate,
$$
\|u(\cdot,t)\|_\infty=
\begin{cases} O(t^{-N/2}),&N\ge3,\\[8pt]
O(t^{-1})&N=1,\\[8pt]
O((t\log t)^{-1})&N=2;
\end{cases}
$$
see Section~\ref{sect:conservation.law} for the statements concerning the two-dimensional case.

Though in low dimensions there is not a residual asymptotic mass, the conservation laws allow  to obtain non-trivial asymptotic quantities, which enter in the characterization of  the large time behavior of the solutions. Thus, for $N=1$ we have a non-trivial limit for the right and left first momenta $M_1^\pm(t)=\int_{\R_\pm} u(x,t)|x|\,dx$. Indeed,
\begin{equation}
\label{eq:limit.momenta.dim1}
M_1^\pm(t)\to  M_{\phi^\pm}^*:=\int_{\R} u_0\phi^\pm\quad\text{as }t\to\infty,
\end{equation}
where $\phi^\pm$ are the solutions to \eqref{eq:stationary} satisfying
$$
(\phi^\pm(x)-\max\{\pm x,0\})\in L^\infty(\mathbb{R}).
$$
In the two-dimensional case, the relevant quantity approaching a non-trivial constant is  the \emph{logarithmic momentum}, $M_{\rm log}(t):=\int_{\mathbb{R}^2}u(x,t)\log|x|\,dx$. Indeed,
$$
M_{\rm log}(t)\to M_\phi^*:=\int_{\mathbb{R}^2} u_0\phi\quad\text{as }t\to\infty,
$$
where $\phi$ is the solution to~\eqref{eq:stationary} satisfying~\eqref{eq:stationary.behavior}; see Section~\ref{sect:conservation.law}.

\

\noindent\textsc{Outer limit. } The non-trivial limit quantities give an indication of the right scalings in order to obtain the asymptotics for the solution  far away from the origin. Thus, for $N\ge3$ the adequate scaling is the one that conserves mass,
$$
u^\lambda(x,t)=\lambda^{N/2}u(\lambda^{1/2} x,\lambda t).
$$
The scaled solution satisfies
$$
\partial_t \ul=L_\l\ul\quad\text{for }x\in(\R^N\setminus\mathcal{H}^\l),\ t>0, \qquad \H^\l=\{x: \l^{1/2} x\in \H\},
$$
where  $L_\l$ is the operator defined by
$$
L_\l \varphi(x)=\l\int_{\R^N} J_\l(x-y)\big(\varphi(y)-\varphi(x)\big)\,dy,\qquad J_\l(x)=\l^{N/2} J(\l^{1/2} x).
$$
If $\varphi\in C^\infty_{\rm c}(\R^N)$,  an easy computation, which uses the symmetry of the kernel plus Taylor's expansion, shows that $L_\l\varphi$ converges uniformly to  $\a\Delta\varphi$ as $\l\to\infty$, with $\a$ as in~\eqref{eq:definition.alpha}.  Hence, the asymptotic behavior is expected to be given by a multiple of the fundamental solution of the \emph{local} heat equation with diffusivity $\a$. Notice that the fundamental solution conserves mass.  The proportionality constant should therefore be given by the limit mass, $M^*_\phi$. Indeed, we have
$$
\lim_{t\to\infty}t^{N/2}\sup\{|u(x,t)-M^*_\phi\Gamma_\a(x,t)|:\,
|x|\ge \delta\sqrt t\}=0 \quad \text{ for  all }\delta>0.
$$
Notice that the only effect of the holes in this outer limit for large dimensions, $N\ge3$, is the loss
of mass.

For $N=1$ the right scaling is the one that preserves the first momentum,
and the asymptotic behavior is given in terms of the \emph{dipole} solution to  the (local) heat equation with diffusivity $\a$,
$$
\mathcal{D}_\a(x,t)=\partial_x \Gamma_\a(x,t)=-\frac{x}{2\a t}\Gamma_q(x,t).
$$
This special solution, which is  self-similar, has $\delta'$, the derivative of the Dirac mass, as initial data,  and preserves the first momentum. Thus,
\begin{equation*}\label{convergence}
\lim_{t\to\infty}t\sup\{|u(x,t)+2M_{\phi^+}^*\mathcal{D}_\a(x,t)|:\,
x\ge \delta\sqrt t\}=0 \quad \text{ for  all }\delta>0,
\end{equation*}
with $M_{\phi^+}^*$ as in~\eqref{eq:limit.momenta.dim1}.
A similar statement holds in sets of the form  $x\le-\delta t^{1/2}$, substituting $M_{\phi^+}^*$ by $-M_{\phi^-}^*$. Notice that in this one-dimensional case the effect of the hole on the outer limit is more dramatic when compared to the case without holes: it changes both the rate of decay and the limit profile. Moreover, we also lose the symmetry of the asymptotic profile, even when the hole is small, compared to the support of the kernel, and hence does not disconnect the domain.

As for the critical two-dimensional case, the fact that the logarithmic momentum has a nontrivial asymptotic limit does not show directly which is the right scaling. However, the \emph{rescaled mass}, $\log t\, M(t)$, behaves for large times as twice the logarithmic momentum (it is here that we use the condition~\eqref{eq:cond.second.momentum}); see~Section~\ref{sect:outer.limit}.  Hence,  the rescaled mass approaches a non-trivial constant, namely
$$
\log t\, M(t)\to 2M_\phi^*.
$$
Therefore, in outer regions the solution is expected to approach a fundamental solution of the local heat equation with variable mass $2M^*_\phi/\log t$. As we will prove in Section~\ref{sect:outer.limit}, this is indeed the case.
\begin{teo}
\label{thm:outer}
\label{thm:far-field.behavior}
Let $N=2$. Assume that $\H$ and $J$ satisfy, respectively,~\eqref{hypotheses.H} and~\eqref{hypotheses.J}. Let $u$ be the solution to~\eqref{problem} with an initial data $u_0$ satisfying~\eqref{eq:hypotheses.u0.1}--\eqref{eq:cond.second.momentum}. Then,
\begin{equation*}
\label{eq:far-field.behavior}
\lim_{t\to\infty}t\log t\sup\left\{\left| u(x,t)-2M_\phi^*\frac{\Gamma_\a(x,t)}{\log t} \right|:\,
|x|\ge \delta\sqrt t\right\}=0 \quad \text{ for  all }\delta>0,
\end{equation*}
where
$\Gamma_\a$ is given by~\eqref{eq:fundamental.solution}, with $\a$ as in~\eqref{eq:definition.alpha}, $M^*_\phi=\int_{\mathbb{R}^2} u_0\phi$, and $\phi$ is the unique solution to~\eqref{eq:stationary}~and~\eqref{eq:stationary.behavior},
\end{teo}
In this borderline dimension the existence of a hole modifies slightly the rate of decay in the outer limit, but it does not change the rescaled asymptotic profile --except for the proportionality constant--, which is still given by the (radially symmetric) profile of the fundamental solution of the heat equation.

It is worth noticing that the mentioned scaling only gives a non-trivial profile if $0<\xi_1\le|x|t^{-1/2}\le \xi_2<\infty$,  in the far field scale. In the \emph{very far field},  $|x|t^{-1/2}\to\infty$, our result says that $t\log t\, u(x,t)\to0$, but not more. Analogous results hold for other dimensions in the very far field.

\

\noindent\textsc{Inner limit. } What happens close to the holes?
For large dimensions, $N\ge 3$, solutions
still decay as $O(t^{-N/2})$ in the inner region
$|x|\le\delta t^{1/2}$ for some $\delta>0$.
If we scale the solutions accordingly,
we get that the new variable $w(x,t)=t^{N/2}u(x,t)$ satisfies
$$
\partial_t w=Lw+\frac{Nw}{2t}.
$$
Thus, $w$ is expected to converge to a stationary solution, solving
\eqref{eq:stationary}. To determine completely this solution we have to prescribe its behavior at infinity. Since there is an overlapping region between the inner and the outer developments,  they can be matched, leading to
$$
t^{N/2}(u(x,t)-M^*_\phi\phi(x)\Gamma_\a(x,t))\to 0\quad\mbox{as }
t\to\infty \mbox{ uniformly in }\R^N.
$$
In particular,
$$
t^{N/2}u(x,t)\to \frac{M^*_\phi\phi(x)}{(4\pi\a)^{N/2}}\quad\mbox{uniformly in compact subsets of }\R^N.
$$
Notice that, as expected, the effect of the hole is more important when we are close to it. In this region, the asymptotic profile does not have radial symmetry, and sees more details of the hole through the function $\phi$.

In low dimensions the determination of the inner behavior is by far more involved. The main reason is that solutions do not decay at the same rate everywhere in sets of the form $|x|\le D t^{1/2}$. Thus, for $N=1$ the rates of decay depend on the ratio $t^{3/2}/|x|$. More precisely, in the \emph{near field} scale, $ |x|\le t^{1/2}h(t)$, $\lim_{t\to\infty}h(t)=0$, the scaled function $t^{3/2}u(x,t)/|x|$ converges to a multiple of $\phi^*(x)/|x|$, where $\phi^*$ is a solution to~\eqref{eq:stationary}--\eqref{eq:stationary.behavior.dim.1} for some constants $b^\pm$. The right choice for the involved constants is, again, obtained through a matching procedure with the outer limit. We obtain
\begin{equation*}
\frac {t^{3/2}}{|x|}\left(u(x,t)+2
\frac{\phi^*(x)}{x}\mathcal{D}_\a(x,t)\right)
\to 0\quad\mbox{as }t\to\infty \mbox{ uniformly in }\R,
\end{equation*}
where $\phi^*$ is the solution to~\eqref{eq:stationary}--\eqref{eq:stationary.behavior.dim.1} with $b^\pm=M^*_{\phi^\pm}$.
In particular,
$$
t^{3/2}u(x,t)\to \frac{\phi^*(x)}{2\a^{3/2}\sqrt{\pi}}\quad\mbox{uniformly in compact subsets of }\R.
$$
Let us remark that in this case the matching is more complicated, since the overlapping region is not so wide as in the case of large dimensions. A main step is obtaining a super-solution giving the right decay rates in the whole inner region and up to the beginning of the far-field scale.

In the two-dimensional case that is the subject of this paper, the rates of decay  depend on the ratio $t(\log t)^2/\log |x|$. Moreover, the scaled function $t(\log t)^2u(x,t)/\log |x|$ converges in the near field scale, $|x|\le t^{1/2}h(t)$, $\lim_{t\to\infty}h(t)=0$, to a multiple of $\phi(x)/\log |x|$, where $\phi$ is the unique solution to~\eqref{eq:stationary} and~\eqref{eq:stationary.behavior}. We finally obtain, after matching the inner and the outer developments, our main result, whose proof is completed in Section~\ref{sect:inner limit}.
\begin{teo}
\label{thm:main}
Under the assumptions of Theorem~\ref{thm:far-field.behavior}, the unique solution $u$ to problem~\eqref{problem} satisfies
\begin{equation}
\label{eq:main.result}
\frac{t(\log t)^2}{\log |x|}\left( u(x,t)-4M^*_\phi\frac{\Gamma_\a(x,t)\phi(x)}{(\log t)^2}  \right)\to 0\quad\mbox{as }t\to\infty \mbox{ uniformly in }\R^2.
\end{equation}
\end{teo}

As a consequence,
$$
\lim_{t\to\infty}t(\log t)^2u(x,t)\to\frac{M^*_\phi}{\pi\a}\phi(x)\quad\mbox{uniformly in compact subsets of }\R^2.
$$

\medskip

\noindent\emph{Remark. } The proof can be easily extended to the case of initial data without sign restrictions.  Indeed, if $u^\pm$ are the solutions with initial data $\{u_0\}_\pm$, then, by the linearity of the equation,  $u=u^+-u^-$. Since $M_\phi^*=\int_{\mathbb{R}^2}\{u_0\}_+\phi-\int_{\mathbb{R}^2}\{u_0\}_-\phi$, the result for general data will follow from the results for $u^+$ and $u^-$. Notice, however, that in the case of initial data with sign changes it may happen that $M_\phi^*=0$.  In this situation our result is not optimal (solutions decay faster), and we should look for a different scaling.

\medskip

Notice that the decay rate $O\left((t\log t)^{-1}\right)$ is seen not only in the far field scale, $|x|\approx\xi t^{1/2}$, $\xi\ne0$, but also  in much smaller \lq parabolas'. Indeed, let $h$ be such that
$h(t)\to0$, $th(t)\to\infty$, $\frac{\log h(t)}{\log t}\to-\alpha$, $0\le\alpha<1$, as $t\to\infty$. Then,
\[
\lim_{t\to\infty}t\log t\,u(x,t)= (1-\alpha)\frac{M^*_\phi}{2\pi\a}\quad\mbox{if }|x|^2=th(t).
\]
In particular, this result holds if $h(t)=t^{-\alpha}$ or $h(t)=(\log t)^{-\gamma}$ with $\gamma>0$ (in which case $\alpha=0$).

%%%%%%%%%%%%%%%%%%%%%%%%%%%%%%%%%%%%%%%%%%%%%%%%%%%%%%%%%5
\section{The stationary problem}
\label{sect:stationary.problem}
\setcounter{equation}{0}

The first aim of this section is to prove the existence of a unique solution $\phi$ to~\eqref{eq:stationary} and~\eqref{eq:stationary.behavior}. Existence will follow from the fact that there exist ordered sub- and super-solutions with adequate behaviors at infinity. We will also obtain  estimates for the first derivatives of $\phi$, which will be used later, in Section~\ref{sect:inner limit}, to determine the inner behavior.

\subsection{Sub- and super-solutions}
We start by constructing sub- and super-solutions to the stationary problem behaving as $\log|x|$ when $|x|\to\infty$.
\begin{lema}\label{lema-subsolution} Let $N=2$, and assume that $\H$ and $J$ satisfy, respectively,~\eqref{hypotheses.H} and~\eqref{hypotheses.J}. There exist constants $R_0,k_0>0$ such that the function
\begin{equation*}\label{subsolution-stationary}
V_-(x)=
\log(|x|^2+k_0)-R_0,\quad x\in \R^2,
\end{equation*}
satisfies
\[-LV_-\le0 \quad\text{in }\R^2,\qquad V_-\le0\quad \text{in }\H.
\]
\end{lema}
\begin{proof}
Since $V_-\in C^\infty(\R^2)$ and the kernel $J$ is radially symmetric, a simple calculation using Taylor's expansion shows that, for $\a$ as in~\eqref{eq:definition.alpha},
$$
LV_-(x)-\a \Delta V_-(x)=\frac1{4!}\sum_{|\beta|=4}\int_{\mathbb{R}^2}J(|x-y|)D^\beta V_-(\xi)(x-y)^\beta,
$$
for some $\xi$ lying in the segment that joins $x$ and $y$. We are using the standard multi-index notation.
A direct computation yields
$$
\Delta V_-(x)=\frac{4 k_0}{(|x|^2+k_0)^2},\qquad
|D^\beta V_-(x)|\le \frac C{(|x|^2+k_0)^2}, \quad|\beta|=4.
$$
Hence, since we only have to take into account the values $y$ such that $|x-y|\le1$,
$$
|D^\beta V_-(\xi)|\le \frac C{(|x|^2+k_0)^2}, \quad|\beta|=4.
$$
Therefore, there exists a constant $K>0$, independent of $k_0$ and $R$, such that
$$
-LV_-(x)\le  \frac{K-4 k_0}{(|x|^2+k_0)^2}.
$$
The first part of our statement follows taking $k_0\ge K/4$.

Once  $k_0$ is fixed, since $\H$ is bounded, by choosing $R_0$ large enough we get $V_-\le0$ in $\H$.
\end{proof}

The construction of a super-solution, which we do next, is far more involved, and much more complicated than for $N\ne2$.
\begin{lema}
\label{lema-supersolution}
Assume the hypotheses of Lemma~\ref{lema-subsolution}. Given $0<r_0<d/4$ such that $B_{2r_0}(0)\subset\H$, there
are constants $\kappa>0$,  $\gamma_0\ge0$, and $D>0$ such that for every $\gamma\ge \gamma_0$ there is a locally bounded super-solution $V_+$ of the stationary problem~\eqref{eq:stationary} satisfying
\begin{equation}\label{V+}
\begin{cases}
-LV_+(x)\ge\frac {\kappa}{|x|^3},& x\in \R^2\setminus \H,\\
V_+(x)\ge \gamma-\gamma_0\ge0,&x\in\mathbb{R}^2,\\
V_+(x)=\log(|x|-r_0)+\gamma ,&|x|\ge D.
\end{cases}
\end{equation}
\end{lema}

\begin{proof} We will prove that there exist $\kappa>0$, $k\ge2$, and a sequence $0= a_1> a_2>\cdots> a_{k+1}$ such that
\begin{equation*}\label{super-solution}
v_+(x)=
\begin{cases}
\log(|x|-r_0) ,&|x|\ge D:=2r_0+k\frac d2,\\[6pt]
a_j,&x\in \Gamma_j:=\left\{D-j\frac d2\le |x|< D-(j-1)\frac d2\right\}, \ j=1,\cdots,k+1,\\[6pt]
a_{k+1},&x\in \Gamma_{k+1}:=B_{2r_0},
\end{cases}
\end{equation*}
satisfies
\begin{equation*}\label{v+}
\quad -Lv_+(x)\ge\frac {\kappa}{|x|^3}\quad\mbox{in }\R^2\setminus B_{2r_0}.
\end{equation*}
Hence, the result will follow just taking $V_+=v_++\gamma$, with
$\gamma\ge \gamma_0:=|a_{k+1}|$.

For $|x|\ge 2r_0$, a direct computation yields
\[ \Delta \big(\log(|x|-r_0)\big)=-\frac{r_0}{|x|(|x|-r_0)^2},\qquad
\big|D^\beta\big(\log(|x|-r_0)\big)\big|\le \frac C{(|x|-r_0)^4},\quad |\beta|=4.
\]
Hence, Taylor's expansion shows that there exists $k\in\mathbb{N}$, $k\ge2$, large enough, such that
\begin{equation}\label{eq-log}
\log(|x|-r_0)-\int_{\mathbb{R}^2}J(x-y)\log(|y|-r_0)\,dy\ge\frac{r_0|x|}{2(|x|-r_0)^4}\quad\mbox{if }|x|\ge \underbrace{2r_0+k\frac{d}2}_{D}-d.
\end{equation}
In particular
\begin{equation}
\label{eq:super.far.away}
-Lv_+(x)\ge \frac{r_0}{2|x|^3}\quad\text{if }|x|\ge D+d.
\end{equation}

The integer $k$ may be chosen so that $\log(|x|-r_0)\ge0$ for $|x|\ge D-d$. Hence, if $0=a_1\ge a_2$, for $D\le |x|\le D+d$ we have, using~\eqref{eq-log},
$$
\begin{array}{rcl}
-Lv_+(x)&=&\displaystyle\log(|x|-r_0)-\int_{\{|x|\ge D\}}J(x-y)\log(|y|-r_0)\,dy-a_2\int_{\Gamma_2}J(x-y)\,dy\\[10pt]
&\ge&\displaystyle \log(|x|-r_0)-\int_{\mathbb{R}^2}J(x-y)\log(|y|-r_0)\,dy\ge\frac{r_0|x|}{2(|x|-r_0)^4}\ge \underbrace{\frac{r_0D}{2(D+d-r_0)^4}}_{c_0}.
\end{array}
$$
We will  prove that this estimate can be extended to the set $2r_0\le|x|\le D$, so that
$$
-Lv_+(x)\ge c_0 \quad\text{if } 2r_0\le|x|\le D+d.
$$
From this inequality and~\eqref{eq:super.far.away} we conclude~\eqref{V+} immediately, just taking $\kappa>0$ small enough.
To prove this claim we start by considering the annulus $\Gamma_1$, and then proceed inductively towards the hole, choosing in each step one of the constants $a_j$. The \emph{positive} constants
$$
\underline J^j=\inf\Big\{\int_{\Gamma_{j+1}}
J(x-y)\,dy:x\in \Gamma_{j}\Big\},\quad j=2,\cdots,k,
$$
will play an important role in the process.

Let $x\in \Gamma_1$. We already have  $0=a_1\ge a_2$, but we are still free to choose $a_2$, as long as it is negative. Then, if $a_3\le0$, since $J$ has unit integral,
$$
\begin{array}{rcl}
-Lv_+(x)&=&\displaystyle-\int_{\{|x|\ge D\}}J(x-y)\log(|y|-r_0)\,dy-a_2\int_{\Gamma_2}J(x-y)\,dy
-a_3\int_{\Gamma_3}J(x-y)\,dy\\[10pt]
&\ge&\displaystyle -\log(D+d-r_0)+|a_2|\,\underline J^2\ge c_0,
\end{array}
$$
if $|a_2|$ is large enough. Notice that $a_2$ has to be strictly negative.

We have already fixed $a_1$ and $a_2$, and have imposed that $a_3\le0$.   If $x\in \Gamma_2$, and $a_4\le0$,
$$
\begin{array}{rcl}
-Lv_+(x)&=&\displaystyle a_2-\int_{\{|x|\ge D\}}J(x-y)\log(|y|-r_0)\,dy+\sum_{i=2}^4 |a_i|\int_{\Gamma_i}J(x-y)\,dy
\\[10pt]
&\ge&\displaystyle a_2 -\log\left(D+\frac{d}2-r_0\right)+|a_3|\,\underline J^3\ge c_0,
\end{array}
$$
if $|a_3|$ is large enough. Since $\underline J^3<1$, we have $|a_3|> |a_2|$.

We now use induction. Let $3\le j\le k-1$. Assume that we have already chosen the constants $0=a_1>a_2>\cdots> a_j$, and imposed $a_{j+1}\le 0$, so that $-Lv_+(x)\ge0$ for $x\in \cup_{i=1}^{j-1}\Gamma_i$. Let now $x\in \Gamma_j$. If $a_{j+2}\le0$,
$$
-Lv_+(x)=\displaystyle a_j+\sum_{i=j-2}^{j+2} |a_i|\int_{\Gamma_i}J(x-y)\,dy
\ge\displaystyle a_j+|a_{j+1}|\,\underline J^{j+1}\ge c_0,
$$
if $|a_{j+1}|$ is large enough. Since $\underline J^{j+1}< 1$, we have $|a_{j+1}|> |a_j|$.

The final step is nearly identical. We have already fixed $\{a_j\}_{j=1}^{k-1}$, and have imposed that $a_{k+1}\le0$. Let $x\in \Gamma_k$. Then,
$$
-Lv_+(x)=\displaystyle a_k+\sum_{i=k-2}^{k+1} |a_i|\int_{\Gamma_i}J(x-y)\,dy
\ge\displaystyle a_k+|a_{k+1}|\,\underline J^{k+1}\ge c_0,
$$
if $|a_{k+1}|$ is large enough. Since $\underline J^{k+1}< 1$, we have $|a_{k+1}|> |a_k|$.
\end{proof}

With the same ideas we can construct another super-solution, that will be of use in Section \ref{section:inner limit}

\begin{lema}\label{stationary-super2} Assume the hypotheses of Lemma~\ref{lema-subsolution}. Let $0<\nu<1$. There exist constants $\kappa>0$, $\gamma_0\ge0$, and $D\ge1$ such that for every $\gamma\ge \gamma_0$ there is a locally bounded super-solution $w_\nu^+$ of the stationary problem~\eqref{eq:stationary} satisfying
\begin{equation}
\label{eq:stationary.super.2}
\begin{cases}
\displaystyle -Lw_\nu^+(x)\ge \frac{\kappa}{|x|^2(\log|x|)^{2-\nu}},&x\in \R^2\setminus \H,\\[6pt]
w_\nu^+(x)\ge \gamma-\gamma_0\ge0,& x\in \R^2,\\[6pt]
w_\nu^+(x)=(\log|x|)^\nu+\gamma,& |x|\ge D.
\end{cases}
\end{equation}
\end{lema}

\subsection{Existence and uniqueness}
The function $\phi$ solving~\eqref{eq:stationary} and \eqref{eq:stationary.behavior} will be obtained as the limit when $n$ tends to infinity of the solutions $\phi_n$ to
\begin{equation*}\label{Pn}
L\phi_n=0\quad\text{in }B_n(0)\setminus\H,\qquad \phi_n=0\quad\text{in }\H,\qquad \phi_n=\frac12V_- \quad\text{in }B_{n+d}(0)\setminus B_n(0),
\end{equation*}
where $V_-$ is the sub-solution constructed  in Lemma \ref{lema-subsolution}.
Existence and uniqueness for such problem is a consequence of \cite[Lemma 3.1]{CEQW}

\begin{prop}\label{existence-stationary} Let $N=2$, and assume that $\H$ and $J$ satisfy, respectively,~\eqref{hypotheses.H} and~\eqref{hypotheses.J}.  There exists a solution to~\eqref{eq:stationary} and \eqref{eq:stationary.behavior}.
\end{prop}
\begin{proof} Let
$V_-$ and $V_+$ be the sub- and super-solutions constructed in Lemmas \ref{lema-subsolution} and \ref{lema-supersolution}.
Taking $\gamma\ge\gamma_0$  large enough, we have $\frac12V_-\le V_+$ in $\R^2$. Thus, the comparison principle gives
\[\frac12V_-\le \phi_n\le V_+ \quad\mbox{in }B_n.
\]

Moreover, by applying the comparison principle once again, we find that $\phi_n\le\phi_{n+1}$ in $B_n$.
In particular, there exists  $\phi=\lim_{n\to\infty}\phi_n$. It is easily seen that this monotone limit solves~\eqref{eq:stationary} and satisfies
\[
\frac12\log(|x|^2+k_0)-\frac12R_0=\frac12V_-(x)\le \phi(x)\le V_+(x)=\log(|x|-r_0)+C,\quad x\in\R^2,
\]
for some constant $C>0$,
which implies~\eqref{eq:stationary.behavior}.
\end{proof}
Uniqueness follows from the fact that the unique bounded solution of~\eqref{eq:stationary} is $\phi\equiv0$.

\begin{prop}
Let $N=2$, and assume that $\H$ and $J$ satisfy, respectively,~\eqref{hypotheses.H} and~\eqref{hypotheses.J}.  The only bounded  solution to \eqref{eq:stationary} is $\phi=0$.
\end{prop}
\begin{proof} The function $\phi_\varepsilon=\phi-\varepsilon V_+$ satisfies $-L\phi_\varepsilon\ge0$ in $\R^2\setminus\H$, and reaches its maximum at some finite point $\bar x$, since by construction $\phi_\varepsilon(x)\to-\infty$ as $|x|\to\infty$.  A standard (for nonlocal operators) argument shows that if $\phi_\varepsilon(\bar x) > 0$ we reach a contradiction. Indeed, if this happens, $\bar x\in \R^2\setminus\H$ and $\phi_\varepsilon$ is constant in $B_{d}(\bar x)$. We can thus propagate the maximum  to the whole connected component of $\R\setminus \H$ where $\bar x$ lies, which leads to a contradiction for points near the boundary of this component, at a distance of it less than $d$.  Then, passing to the limit as $\varepsilon\to0$, we obtain $\phi\le0$. The same argument applied to $-\phi$ leads to  $\phi \ge 0$.
\end{proof}

\medskip

\noindent\emph{Remark. } This result is also true for $N=1$; see~\cite{CEQW3}. However, it fails for $N\ge3$. Indeed, in large dimensions there are stationary solutions that take a constant value, different from zero, at infinity; see~\cite{CEQW}.

\subsection{Estimates for the derivatives}
To prove them, we use that $\phi$ solves a problem of the form
\begin{equation}
\label{eq:inhomogeneous}
\partial_t u-L u=f\quad\text{in }\R^2\times\R_+, \qquad u(x,0)=u_0(x),\quad x\in\R^2.
\end{equation}
By the variations of constants formula, solutions to~\eqref{eq:inhomogeneous} can be written in terms of the fundamental solution $F=F(x,t)$ of the operator $\partial_t-L$ in the whole space, which
can be decomposed as
\begin{equation*}\label{fund-sol}
F(x,t)=\textrm{e}^{-t}\delta(x)+W(x,t),
\end{equation*}
where $\delta$ is the Dirac mass at the origin and $W$ is a nonnegative smooth function defined via its Fourier transform,
\begin{equation*}
\label{eq:transform.omega}
\widehat W(\xi,t)=\textrm{e}^{-t}\left(\textrm{e}^{\hat J(\xi)t}-1\right);
\end{equation*}
see~\cite{CCR}. Thus,
\begin{equation*}
\label{eq:representation.formula}
\begin{aligned}
u(x,t)=&\textrm{e}^{-t}u_0(x)+\int_{\mathbb{R}^2} W(x-y,t)u_0(y)\,dy\\
&+\int_{0}^t\textrm{e}^{-(t-s)}f(x,s)\,ds+\int_{0}^t\int_{\mathbb{R}^2} W(x-y,t-s)f(y,s)\,dyds.
\end{aligned}
\end{equation*}
Therefore, estimates for solutions to~\eqref{eq:inhomogeneous}, and in particular for $\phi$, will follow if we have good estimates for the right hand side of the equation,  $f$,  and  for the regular part, $W$, of the fundamental solution.

The asymptotic convergence
of $W$ to the fundamental solution of the local heat equation
with diffusivity $\a$ yields a first class of estimates. Indeed,  for every multi-index $\beta\in\mathbb{N}^2$ there is a constant $C$ such that
\begin{equation}\label{estima-W}
|D_x^\beta
W(x,t)-D_x^\beta\Gamma_\a(x,t)|\le Ct^{-\frac{|\beta|+3}2},\quad x\in\mathbb{R}^2,\ t>0;
\end{equation}
see \cite{IR1}. Hence,  in
particular,
\begin{equation}\label{decaimiento-W}
|D_x^\beta W(x,t)|\le Ct^{-\frac{|\beta|+2}2}, \quad x\in\mathbb{R}^2,\ t>0.
\end{equation}
These estimates give the right order of time decay, and will prove to be useful later, in Sections~\ref{sect:outer.limit} and~\ref{sect:inner limit}. However,  they do not take into account the spatial structure of $W$, and are quite poor for $t\approx 0$;  hence they are not enough for our present goal.
Instead, we will use the pointwise estimate
\begin{equation}\label{eq:pointwise.estimate.W}
|D_x^\beta W(x,t)|\le \frac{Ct}{1+|x|^{4+|\beta|}}, \quad x\in\mathbb{R}^2,\ t>0,
\end{equation}
and the integral estimate
\begin{equation}
\label{eq:integral.estimate.W}
\int_{\R^2}|D_x^\beta W(x,t)|\,dx\le C t^{-|\beta|/2}, \quad x\in\mathbb{R}^2,\ t>0,
\end{equation}
both of them valid for all $\beta\in\mathbb{N}^2$. These estimates were proved in~\cite{TW} through a comparison argument, using that $W$ is a solution to
\begin{equation}\label{eq-W}
\begin{cases}
\partial_t W(x,t)-LW(x,t)=e^{-t}J(x)\quad&\mbox{in }\mathbb{R}^2\times\R_+,\\
W(x,0)=0\quad&\mbox{in }\mathbb{R}^2.
\end{cases}
\end{equation}

\begin{lema}
\label{lemma:estimates.derivatives.phi} Assume the hypotheses in Proposition~\ref{existence-stationary}.  Let $\phi$ be the solution to~\eqref{eq:stationary} and~\eqref{eq:stationary.behavior}.
There exists a constant $C>0$
such that
\begin{equation*}\label{psi}
|\nabla\phi(x)|\le \frac {C}{|x|} \quad\mbox{in }\R^2\setminus\H.
\end{equation*}
\end{lema}
\begin{proof}
The result will follow from an analogous estimate for $\psi(x)=\phi(x)-g(x)$, where $g(x)=\frac12\log(1+|x|^2)$.
The function $\psi$, which is bounded, is a solution to~\eqref{eq:inhomogeneous} with a right hand side
$$
f(x)=-\underbrace{\X_\H(x)(J*\phi)(x)}_{f_1(x)}-\underbrace{Lg(x)}_{f_2(x)},
$$
and initial data $u_0(x)=\psi(x)$.

Let $\psi^\l(x)=\psi(\l x)$. By differentiating the equation for $\psi$ given by the variations of constants formula, we get for $|x|\ge1/2$ and $\l$ large enough so that $\l x\notin\H$,
$$
\begin{aligned}
\nabla\psi^\l(x)=&\underbrace{\frac\l{1-\textrm{e}^{-t}}\int_{\R^2} \nabla W(\l x-y,t)\psi(y)\,dy}_{\mathcal A(x,t)}\\
&-\underbrace{\frac\l{1-\textrm{e}^{-t}}
\int_0^t\int_{\H}
\nabla W(\l x-y,t-s)(J*\phi)(y)\,dy\,ds}_{\mathcal B(x,t)}\\
&-\underbrace{\frac\l{1-\textrm{e}^{-t}}\int_0^t\int_{\R^2}
\nabla W(\l x-y,t-s)f_2(y)\,dy\,ds}_{\mathcal C(x,t)}-\underbrace{\l \nabla f_2(\l x)}_{\mathcal D(x,t)}.
\end{aligned}
$$
The first term is easily estimated using~\eqref{eq:integral.estimate.W} with $|\beta|=1$, since $\psi$ is bounded,
$$
|{\mathcal A}(x,t)|\le \frac{C\l t^{-1/2}}{1-e^{-t}}\le C\quad \text{if we take }t=\l^2,\ \l\ge1.
$$

As for the second term, the key is that, since $\H$ is bounded,  there exists $\l_0$ such that for $y\in\H$ and $\l\ge\l_0$ there holds that $|\l x-y|\ge \frac12\l|x|$. Hence, using that $J*\phi$ is bounded in $\H$, and estimate~\eqref{eq:pointwise.estimate.W} with $|\beta|=1$, we have
$$
|{\mathcal B}(x,t)|\le C|\H|\frac\l{1-e^{-t}}\frac{2^5}{\l^5|x|^5}\int_0^t(t-s)\,ds= \frac{C\l^{-4} t^2}{(1-e^{-t})|x|^5}\le C\quad\text{if }|x|\ge 1/2,\ \l\ge\l_0, \ t=\l^2.
$$

In order to control the third term we decompose the integral in two parts and use the bound $|f_2(x)|\le \frac C{(1+|x|^2)^2}$, which was obtained in the course of the proof of Lemma \ref{lema-subsolution}. We have,
\[
\begin{aligned}
|{\mathcal C}(x,t)|\le&\underbrace{\frac\l{1-e^{-t}}\int_0^t\int_{\{|y|\le\frac\l2|x|\}}|\nabla W(\l x-y,t-s)|\frac{dy}{(1+|y|^2)^2}\,ds}_{\mathcal C_1(x,t)}\\
&+
\underbrace{\frac\l{1-e^{-t}}\int_0^t\int_{\{|y|\ge\frac\l2|x|\}}|\nabla W(\l x-y,t-s)|\frac{dy}{(1+|y|^2)^2}\,ds}_{\mathcal C_2(x,t)}.
\end{aligned}
\]
The first integral is easily controlled thanks to~\eqref{eq:pointwise.estimate.W} with $|\beta|=1$,
$$
{\mathcal C}_1(x,t)\le \frac {C\l}{1-e^{-t}}\frac{t^2}{\l^5|x|^5}\int_{\{|y|\le\frac\l2|x|\}}\frac{dy}{(1+|y|^2)^2}
\le C\quad\text{if }|x|\ge\frac12,\ \l\ge1,\text{ and }t=\l^2.
$$
On the other hand, using the estimates~\eqref{decaimiento-W} and~\eqref{eq:pointwise.estimate.W} with $|\beta|=1$, we get
\[\begin{aligned}
{\mathcal C}_2(x,t)&\le C\frac\l{1-e^{-t}}\Big(\int_0^{t-1}(t-s)^{-3/2}\,ds+\int_{t-1}^t (t-s)\,ds\Big)\int_{\{|y|\ge\frac\l2|x|\}}\frac{dy}{(1+|y|^2)^2}\\
&\le C\frac{\l^{-1}}{(1-e^{-t})|x|^2}\le C\quad\text{if }|x|\ge \frac12,\ \l\ge1,\ t=\lambda^2.
\end{aligned}\]

We finally estimate $\mathcal{D}$. To this aim we notice that $\partial_{x_i} {f_2}(x)=L\left(\frac{x_i}{1+|x|^2}\right)$. Hence, since
$$
\left|\Delta \left(\frac{x_i}{1+|x|^2}\right)\right|\le \frac C{(1+|x|^2)^2},\qquad \left|D^\beta\left(\frac{x_i}{1+|x|^2}\right)\right|\le \frac C{(1+|x|^2)^{2}},\quad |\beta|=4,
$$
using Taylor's expansion we
get $|\partial_{x_i} {f_2}(x)|\le \frac C{(1+|x|^2)^{2}}$, which implies
\[
|\mathcal D(x,t)|\le \frac{C\l}{(1+\l^2|x|^2)^2}\le C\quad\mbox{if }|x|\ge1/2.
\]

In conclusion, $|\nabla\psi^\l(x)|\le C$ if $|x|\ge\frac12$ and $\l\ge\l_0$. This yields, by taking $|y|\ge \l_0$, $x=y/|y|$ and $\l=|y|$,
\[
|y||\nabla\psi(y)|=|\nabla\psi^\l(x)|\le C.
\]
Recalling the definition of $\psi$, we see that this implies that
\[
|\nabla \phi(y)|\le \frac C{|y|} \quad\mbox{if }|y|\ge \l_0.
\]

The lemma is proved, since $\phi\in C^1(\R^2\setminus\H)$.
\end{proof}

\medskip

\noindent\emph{Remark. } As a corollary of Lemma~\ref{lemma:estimates.derivatives.phi} we have that
\begin{equation}
\label{eq:differences.phi}
|\phi(y)-\phi(x)|\le \frac C{|x|}\quad\mbox{if }x\in\mathbb{R}^2\setminus\H, \ |x-y|<d.
\end{equation}
Indeed, applying the Mean Value Theorem,
\[
|\phi(y)-\phi(x)|\le\frac{Cd}{|\xi|}\le\frac{Cd}{|x|-d}\le \frac {2Cd}{|x|}\quad\mbox{if }x\in\mathbb{R}^2\setminus\H,\  |x|\ge 2d, \ |x-y|<d.
\]
The estimate for $|x|\le 2d$ is immediate, since $\phi$ is locally bounded.

%%%%%%%%%%%%%%%%%%%%%%%%%%%%%%%%%%%%%%%%%%%%%%%%%%%%%%%%%
\section{A conservation law and some by-products}
\label{sect:conservation.law}
\setcounter{equation}{0}

Solutions to~\eqref{problem} satisfy a conservation law that will allow us to prove that the mass has a logarithmic decay rate. This is later used to obtain the decay rate of solutions, a super-solution that gives the right rates of decay in inner regions, the asymptotic limit of the \lq logarithmic'-momentum, and bounds for the first and second momenta when~\eqref{eq:cond.second.momentum} holds.
\begin{prop}
\label{prop:conservation.law}
Let $N=2$. Assume that $\H$, $J$ and $u_0$ satisfy~\eqref{hypotheses.H},~\eqref{hypotheses.J} and~\eqref{eq:hypotheses.u0.1} respectively. Let $\phi$ be the solution to \eqref{eq:stationary} and~\eqref{eq:stationary.behavior}, and $u$ the solution to problem~\eqref{problem}. Then,
\begin{equation}
\label{eq:conservation.law}
M_\phi(t):=\int_{\mathbb{R}^2} u(x,t)\phi(x)\,dx=\underbrace{\int_{\mathbb{R}^2} u_0\phi}_{M_\phi^*}\quad\mbox{for every } t\ge0.
\end{equation}
\end{prop}
The proof, based on Fubini's Theorem, is entirely similar to the one for dimensions N=1 or $N\ge3$; see \cite{CEQW,CEQW3}.  Hence, we omit it.

A first consequence of this conservation law is a bound for the mass at time $t$ which in particular shows that the mass decays to 0 as $t\to\infty$.

\begin{prop}\label{prop-mass-decay} If in addition to the assumptions of Proposition~\ref{prop:conservation.law} we have also $u_0\in  L^\infty(\R^2)$, then there exist $t_0>0, C_0>0$ such that
\begin{equation}\label{mass-estimate}
M(t):=\int_{\mathbb{R}^2} u(x,t)\,dx\le \frac {C_0}{\log t}\quad\mbox{if }t\ge t_0.
\end{equation}
\end{prop}
\begin{proof}
The solution, $u_C$, of the Cauchy problem with initial condition $u_0$ is a super-solution of the Cauchy-Dirichlet problem~\eqref{problem}, since $u_C(\cdot,t)>0$ for $t>0$. Therefore,  the asymptotic behavior of $u_C$, proved in \cite{IR1}, provides an estimate of the time decay rate of $u$: $u(x,t)\le u_C(x,t)\le C t^{-1}$. Though this rate is not optimal, as will turn out later, it is enough to get the decay rate of the mass. Indeed, let $\sigma>0$ such that $\log|x|\le \sigma\phi(x)$ for $x\notin\H$ (recall that $\phi>0$ on $\partial\H$). Then,
$$
M(t)\le\int_{|x|^2\le \delta(t) t}u(x,t)\,dx+\sigma\int_{|x|^2\ge \delta(t) t}\frac{u(x,t)\phi(x)}{\log|x|}\,dx\le C t^{-1}\delta t+\frac\sigma{\frac12\log(\delta t)}M^*_\phi
$$
if $\delta(t) t\ge \mathcal{R}^2$, with $\mathcal{R}$  as in~\eqref{hypotheses.H}.
Taking $\delta(t)=1/\log t$ we get that both terms on the right hand side of this estimate almost balance, and \eqref{mass-estimate}  follows.
\end{proof}

\medskip

\noindent\emph{Remark. } We will see later that $\log t M(t)$ has a nontrivial limit as $t\to\infty$. Thus, the decay rate in Proposition \ref{prop-mass-decay} is optimal.

\medskip

With this estimate for the decay rate of the mass we can improve the decay rate of $u(\cdot, t)$.

\begin{prop}\label{prop-decay-rate} Under the assumptions of Proposition~\ref{prop-mass-decay}, there exists a constant $C_1>0$ such that
\begin{equation}\label{decay-rate}
u(x,t)\le \frac {C_1}{t\log t}\quad\mbox{if } t\ge t_0.
\end{equation}
\end{prop}
\begin{proof} We estimate  $u$ at time $t$ by the solution of the Cauchy problem with initial data $u(x,t/2)$. Thus, using the bound~\eqref{decaimiento-W} with $|\beta|=0$ for $W$, and the bound for the mass \eqref{mass-estimate}, we get
\[\begin{aligned}
u(x,t)&\le u_C(x,t)=e^{-t}u(x,t/2)+\int_{\mathbb{R}^2} W(x-y,t)u(y,t/2)\,dy\\
&\le e^{-t}\|u_0\|_{L^\infty(\R^2)}+Ct^{-1}M(t/2)\le \frac {C}{t\log t}\quad\mbox{if } t\ge t_0.
\end{aligned}
\]
\end{proof}

As we will see in Section~\ref{sect:outer.limit}, the bound provided by~\eqref{decay-rate} gives the optimal global decay rate for the solution. However, $u$ decays faster than $O\left((t\log t)^{-1}\right)$ in inner regions.  Once we have at hand the global decay rate, this last assertion is proved by means of comparison with a suitable super-solution that we construct next. A similar super-solution will be used later in the study of the inner limit. With that application in mind, we keep a parameter $\gamma\ge 2$ that will be set equal to 2 in the estimate of $u$.

\begin{lema} \label{lema:super1} Let $0<a<\a=\frac14\int_{\mathbb{R}^2}J(z)|z|^2\,dz$, and let $V_+$ be the super-solution (depending on the radius $r_0>0$)  to the stationary problem constructed in Lemma \ref{lema-supersolution}. There exist $b,T>0$ such that the function
\begin{equation*}\label{evolutionary-supersolution}
V(x,t)=\frac{\Gamma_a(x,t)\big(V_+(x)+b\big)}{\log^\gamma t},\qquad \Gamma_a(x,t)=\frac{e^{-\frac{|x|^2}{4at}}}{4\pi at},
\end{equation*}
satisfies
\begin{equation}
\label{eq:estimate.evolutionary.supersolution}
(\partial_t V-LV)(x,t)\ge\frac1{10}\Big(\frac \a a-1\Big)\frac{\Gamma_a(x,t)\big(V_+(x)+b\big)}{t\log^\gamma t}\quad\mbox{in }4r_0^2\le|x|^2\le 2at,\quad t\ge T.
\end{equation}
\end{lema}
\begin{proof} We have
\begin{equation}
\label{eq:decomposition.nonlocal.heat.V}
\begin{array}{rcl}
\displaystyle(\partial_t V-LV)(x,t)&=&\displaystyle\frac{V_+(x)+b}{\log^\gamma t}\left(\partial_t\Gamma_a-L\Gamma_a\right)(x,t)\\[10pt]
&&\displaystyle-\gamma\frac{\Gamma_a(x,t)(V_+(x)+b)}{t\log^{\gamma+1} t}-\frac{\Gamma_a(x,t)LV_+(x)}{\log^\gamma t}
\\[10pt]
&&\displaystyle -\underbrace{\frac{1}{\log^\gamma t}
\int_{\mathbb{R}^2} J(x-y)(\Gamma_a(y,t)-\Gamma_a(x,t))(V_+(y)-V_+(x))\,dy}_{\mathcal{A}(x,t)}.
\end{array}
\end{equation}

In order to estimate $\partial_t\Gamma_a-L\Gamma_a$, we observe that Taylor's expansion yields
$$
L\Gamma_a(x,t)=\a \Delta \Gamma_a(x,t)+\frac1{4!}\sum_{|\beta|=4}\int_{\mathbb{R}^2}J(|x-y|)D_x^\beta \Gamma_a(\xi,t)(x-y)^\beta\,dy
$$
for some $\xi$ lying in the segment that joins $x$ and $y$. Hence, since $\partial_t\Gamma_a=a\Delta\Gamma_a$,
$$
(\partial_t\Gamma_a-L\Gamma_a)(x,t) = \Big(1-\frac \a a\Big)\partial_t\Gamma_a(x,t)-\frac1{4!}\sum_{|\beta|=4}\int_{\mathbb{R}^2}J(|x-y|)D_x^\beta \Gamma_a(\xi,t)(x-y)^\beta\,dy.
$$
To deal with  the main term, we compute
$$
\displaystyle\partial_t\Gamma_a(x,t)=\left(\frac{|x|^2}{4at}-1\right)\frac{\Gamma_a(x,t)}t\le -\frac{\Gamma_a(x,t)}{2t}<0 \quad\mbox{if }|x|^2\le 2a t.
$$
An estimate for the error term  follows from
$$
\displaystyle|D^\beta_x\Gamma_a|(\xi,t)\le\frac{C\Gamma_a(\xi,t)}{t^2}\le \frac{C\Gamma_a(x,t)}{t^2} \quad\mbox{if } |\beta|=4,\ |x|^2\le 2a t.
$$
Thus, since $0<a<\a$, for some $t_0$ large enough we have the lower bound
\begin{equation}
\label{eq:estimate.nonlocal.heat.Gamma}
(\partial_t\Gamma_a-L\Gamma_a)(x,t) \ge \frac12\left(\frac \a a-1\right)\frac{\Gamma_a(x,t)}{t}-\frac{C\Gamma_a(x,t)}{t^2}
\ge \frac14\Big(\frac \a a-1\Big)\frac{\Gamma_a(x,t)}{t},\quad t\ge t_0.
\end{equation}

Let us now bound the integral term $\mathcal{A}(x,t)$. On the one hand,
$$
|V_+(y)-V_+(x)|\le \log(D+d-r_0)-a_{k+1}\quad\text{if }x,y\in B_{D+d}(0).
$$
On the other hand, by the Mean Value Theorem,
$$
|V_+(y)-V_+(x)|=\left|\nabla V_+(\xi)\cdot(y-x)\right|\le \frac{d}{|\xi|-r_0}\le\frac{d}{|x|-d-r_0}\le\frac{C}{|x|}\quad\text{if }x,y\in\R^2\setminus B_D(0).
$$
The difference between the values of $\Gamma_a$ at $x$ and $y$ is also estimated thanks to the Mean Value Theorem, using that
$$
|\nabla\Gamma_a(\xi,t)|=\frac{|\xi|\,\Gamma_a(\xi,t)}{2at}\le \frac{C|x|\,\Gamma_a(x,t)}{t}.
$$
Gathering all this information, we obtain
\begin{equation}
\label{eq:estimate.integral.term}
|\mathcal{A}(x,t)|
\le C_{d, a}\frac{\Gamma_a(x,t)}{t\log^\gamma t}\quad\text{for }4r_0^2\le |x|^2\le 2at.
\end{equation}

If we use estimates~\eqref{eq:estimate.nonlocal.heat.Gamma} and \eqref{eq:estimate.integral.term}, together with~\eqref{V+}, in~\eqref{eq:decomposition.nonlocal.heat.V}, we get
$$
\begin{aligned}
(\partial_t V-LV)(x,t)&\ge \frac{\Gamma_a(x,t)}{t\log^\gamma t}\left(\frac14\Big(\frac \a a-1\Big)\big(V_+(x)+b\big)-\gamma\frac{V_+(x)+b}{\log t}-C_{d, a}\right)\\
&\ge\frac1{10}\Big(\frac \a a-1\Big)\frac{\Gamma_a(x,t)\big(V_+(x)+b\big)}{t\log^\gamma t}
\end{aligned}
$$
if $b$ is large enough and $t\ge T$, with $T$ large.
\end{proof}

Now we can obtain the time decay rates in all near field scales.

\begin{coro}\label{prop:improved size estimate}
Let $0<a<\a$. Under the assumptions of Proposition~\ref{prop-mass-decay}, there exist constants $C,b,T>0$ such that
\[
u(x,t)\le C\frac{e^{-\frac{|x|^2}{4at}}(V_+(x)+b)}{t(\log t)^2}\quad\mbox{if } |x|^2\le 2at,\quad t\ge T.
\]
\end{coro}
\begin{proof}
Once we have established the global decay rate~\eqref{decay-rate}, the result follows by comparison with a large enough multiple of  the super-solution $V$ constructed in Lemma \ref{lema:super1} with $\gamma=2$, taking $T\ge t_0$,  since
\[
V(x,t)\ge\frac{\kappa}{t\log t}\quad\mbox{if } 2at\le |x|^2\le 2at+d
\]
for a certain constant $\kappa>0$, and
\[
V(x,T)\ge c>0 \quad\mbox{if } |x|^2\le 2aT.
\]
\end{proof}

A fourth by-product of the conservation law  is the large time behavior of the \lq\lq logarithmic momentum''.
\begin{coro}
\label{cor:ltb.log.momentum}
Under the assumptions of Proposition~\ref{prop-mass-decay},
$$
\lim_{t\to\infty} \int_{\mathbb{R}^2}u(x,t)\log|x|\,dx=M^*_\phi.
$$
\end{coro}
\begin{proof}
Using the conservation law~\eqref{eq:conservation.law}, together with the mass decay estimate~\eqref{mass-estimate} and the estimate~\eqref{eq:stationary.behavior} for the stationary solution, we immediately obtain
$$
\left|\int_{\mathbb{R}^2}u(x,t)\log|x|\,dx-M^*_\phi\right|\le
\int_{\mathbb{R}^2}u(x,t)|\log|x|-\phi(x)|\le \frac {C}{\log t}\quad\mbox{if } t\ge t_0,
$$
from where the result follows.
\end{proof}

As mentioned in the Introduction, if $u_0$ satisfies~\eqref{eq:cond.second.momentum}, the second momentum of $u(\cdot,t)$ is finite for all times. We now obtain an estimate for its growth, which follows from the mass estimate~\eqref{mass-estimate}. This estimate for the second momentum will be used later to study the limit of the rescaled mass.
\begin{lema}\label{lema:second-momentum} If, in addition to the assumptions of Proposition~\ref{prop-mass-decay}, $u_0$ satisfies~\eqref{eq:cond.second.momentum}, then there are constants $t_1>0$ and $C_2>0$ such that
\begin{equation}\label{second-momentum}
M_2(t):=\int_{\mathbb{R}^2} u(x,t)|x|^2\,dx\le C_2\frac t{\log t}\quad\mbox{for }t\ge t_1.
\end{equation}
\end{lema}
\begin{proof}
There holds,
\[\begin{aligned}
M_2'(t)&=\int_{\mathbb{R}^2} \partial_t u(x,t)|x|^2\,dx=\int_{\R^2\setminus\H} Lu(x,t)|x|^2\,dx\\
&= \int_{\R^2} Lu(x,t)|x|^2\,dx-\int_\H\int_{\mathbb{R}^2} J(x-y)u(y,t)|x|^2\,dy\,dx\\
&\le \int_{\R^2} Lu(x,t)|x|^2\,dx=\int_{\mathbb{R}^2} u(x,t) L|x|^2\,dx\\
&=c\int_{\mathbb{R}^2} u(x,t)\le \frac{cC_0}{\log t}\quad\mbox{for }t\ge t_0.
\end{aligned}
\]
 Integrating this inequality, we obtain \eqref{second-momentum}, for some $t_1\ge t_0$.
\end{proof}
The control of the growth of the second momentum yields in turn a control on the growth of the first momentum.
\begin{lema}\label{lema:first-momentum} Under the assumptions of Lemma~\ref{lema:second-momentum},
\begin{equation}\label{first-momentum}
M_1(t):=\int_{\mathbb{R}^2} u(x,t)|x|\,dx\le C_3\frac {t^{1/2}}{\log t}\quad\mbox{for }t\ge t_1.
\end{equation}
\end{lema}
\begin{proof}
Using~\eqref{mass-estimate} and~\eqref{second-momentum}, we get
$$
\begin{array}{l}
M_1(t)\le\displaystyle t^{1/2}\int_{|x|\le t^{1/2}}u(x,t)\,dx+\frac{1}{t^{1/2}}\int_{|x|\ge t^{1/2}}|x|^2u(x,t)\, dx\le \displaystyle \frac{C_0t^{1/2}}{\log t}+\frac{C_2 t^{1/2}}{\log t}.
\end{array}
$$
\end{proof}

%%%%%%%%%%%%%%%%%%%%%%%%%%%%%%%%%%%%%%%%%%%%%%%%%%
\section{Outer limit}
\label{sect:outer.limit}
\setcounter{equation}{0}

We have now all the necessary tools to obtain the outer limit, Theorem~\ref{thm:outer}.
The first step is to determine the asymptotic value of the rescaled mass in terms of the initial condition.
\begin{prop}
Under the assumptions of Theorem~\ref{thm:main},
\begin{equation*}\label{final-mass}
\log t\int_{\mathbb{R}^2} u(x,t)\,dx\to 2M_\phi^*\quad\mbox{as }t\to\infty.
\end{equation*}
\end{prop}
The result follows immediately from Corollary~\ref{cor:ltb.log.momentum}, combined with the following lemma, that states that the scaled mass behaves for large times as twice the logarithmic momentum.
\begin{prop}
Under the assumptions of Theorem~\ref{thm:main},
$$
\lim_{t\to\infty}\log t\int_{\mathbb{R}^2} u(x,t)\,dx=\lim_{t\to\infty}2\int_{\mathbb{R}^2} u(x,t)\log|x|\,dx.
$$
\end{prop}
To prove this proposition we consider separately the regions $\{|x|^2\le t\log t\}$ and $\{|x|^2\ge t\log t\}$.
\begin{lema}
Under the assumptions of Proposition~\ref{prop-mass-decay},
$$
\underbrace{\int_{\{|x|^2\le t\log t\}}u(x,t)\left(\log t-\log|x|^2\right)\,dx}_{\mathcal{A}(t)}=O\left(\frac{\log\log t}{\log t}\right)\quad\text{as }t\to\infty.
$$
\end{lema}

\begin{proof} We have
$$
 |\mathcal A(t)|\le
 \underbrace{\int_{\left\{|x|^2\le \frac{t}{\log t}\right\}} u(x,t)\left|\log\frac{|x|^2}t\right|\,dx}_{\mathcal A_1(t)}
  +\underbrace{\int_{\left\{\frac{t}{\log t}\le|x|^2\le t\log t\right\}} u(x,t)\left|\log\frac{|x|^2}t\right|\,dx}_{\mathcal A_2(t)}.
$$
Notice that
$\left|\log\frac{|x|^2}t\right|\le C\log t$ in $\{4\le |x|^2\le t/\log t,\, t\ge \textrm e\}$. Therefore, since $u(x,t)=0$ for $|x|\le 2$,
 using the size estimate~\eqref{decay-rate} we get
$$
|\mathcal{A}_{1}(t)|\le \frac{C \log t}{t\log t}\int_{\left\{|x|^2\le \frac{t}{\log t}\right\}}\,dx=\frac{C}{\log t},\quad t\ge\max\{t_0,\textrm{e}\}.
$$
Using now that $\left|\log\frac{|x|^2}t\right|\le C\log\log t$ in $\{t/\log t\le |x|^2\le t\log t,\, t\ge \textrm e\}$, together with the decay of the mass~\eqref{mass-estimate}, we obtain
$$
|\mathcal{A}_2(t)|\le  C\log\log t\int_{\mathbb{R}^2} u(x,t)\,dx\le \frac{C\log\log t}{\log t},\quad t\ge\max\{t_0,\textrm{e}\}.
$$
\end{proof}

In the case of the second region, both the integral corresponding to the mass and the one corresponding to the logarithmic momentum converge separately to 0.
\begin{lema}
\label{lema:colas}
Under the assumptions of Lemma~\ref{lema:second-momentum},
$$
\begin{array}{l}
\displaystyle\log t\int_{\{|x|^2\ge t\log t\}}u(x,t)\,dx=O\left(\frac1{\log t}\right),\\[10pt]
\displaystyle\int_{\{|x|^2\ge t\log t\}}u(x,t)\log |x|^2\,dx=O\left(\frac1{\log t}\right),
\end{array}
\qquad \text{as }t\to\infty.
$$
\end{lema}

\begin{proof}
First, using Lemma~\ref{lema:second-momentum},
$$
0\le\log t\int_{\{|x|^2\ge t\log t\}}u(x,t)\,dx\le\frac{\log t}{t\log t}\int_{\mathbb{R}^2} u(x,t)|x|^2\,dx
\le \frac{C_2}{\log t}\quad t\ge\max\{t_1,1\}.
$$

As for the second integral, we use that the function $r\to \frac r{\log r}$ is nondecreasing for $r\ge\textrm{e}$. Hence, taking  $t_2$ such that $t_2\log t_2=\textrm{e}^2$, and using again Lemma~\ref{lema:second-momentum}, we have
$$
\begin{array}{rcl}
\displaystyle\int_{\{|x|^2\ge t\log t\}}u(x,t)\log|x|^2\, dx&=&\displaystyle\int_{\{|x|^2\ge t\log t\}}u(x,t)\frac{|x|^2}{\frac{|x|^2}{\log |x|^2}}\,dx\le \frac{\log (t\log t)}{t\log t}\int_{\mathbb{R}^2} u(x,t)|x|^2\,dx\\[10pt]
&\le& \displaystyle {C_2}
\frac{\log (t\log t)}{t\log t}\frac t{\log t}\le \frac{C}{\log t},\quad t\ge\max\{t_1,t_2\}.
\end{array}
$$
\end{proof}
Once we have identified the limit of the rescaled mass, we can proceed to obtain the outer limit.
Let us remark that, although we are dealing with the outer region, the proof requires the refined bound in the inner region provided by~Corollary~\ref{prop:improved size estimate}.

The idea is that $\log t\, u(\cdot,t)\approx v(\cdot,t)$ for large times in the outer region, where
$v=v(x,s)$ is the solution to
\begin{equation}
\label{eq:def.v}
\partial_s v-\a\Delta v=0,\ x\in\R^2,\ s>\tau(t):=t/(\log t)^{1/2},\quad
v\big(x,\tau(t)\big)=\log \tau(t)u\big(x,\tau(t)\big), \ x\in \R^2.
\end{equation}
\begin{lema}
Under the assumptions of Theorem~\ref{thm:outer}, the solution $u$ to problem~\ref{problem} satisfies
\begin{equation*}
\label{eq:far-field.behavior.u.v}
\lim_{t\to\infty}t\sup\left\{\left|\log t\,u(x,t)-v(x,t) \right|:\,
|x|\ge \delta\sqrt t\right\}=0 \quad \text{ for  all }\delta>0,
\end{equation*}
where $v$ is the solution to~\eqref{eq:def.v}.
\end{lema}

\begin{proof}
Since $u$ is a solution to the Cauchy problem \eqref{eq:inhomogeneous} with right hand side   $f=-\chi_\H(J*u)$, and initial data $u(\cdot,\tau(t))$ at time $s=\tau(t)$, then
$$
\begin{array}{l}
t\big(\log t\, u(x,t)-v(x,t)\big)=\underbrace{t\big(\log t-\log \tau(t)\big)u(x,t)}_{\mathcal{A}(x,t)}+\underbrace{t\log \tau(t)\,\textrm{e}^{-(t-\tau(t))}u(x,\tau(t))}_{\mathcal B(x,t)}\\
\qquad\qquad\qquad+\underbrace{t\log \tau(t)\int_{\R^2} \big(W(x-y,t-\tau(t))-\Gamma_\a(x-y,t-\tau(t))\big)u(y,\tau(t))\,dy}_{\mathcal C(x,t)}\\
\qquad\qquad\qquad-\underbrace{t\log \tau(t)\int_{\tau(t)}^t e^{-(t-s)}\chi_\H(x)(J*u(\cdot,s))(x)\,ds}_{\mathcal D(x,t)}\\
\qquad\qquad\qquad-\underbrace{t\log \tau(t)\int_{\tau(t)}^t\int_\H W(x-y,t-s)(J*u(\cdot,s))(y)\,dy\,ds}_{\mathcal E(x,t)}.
\end{array}
$$
We have, using~\eqref{decay-rate}, the $L^\infty$ bound in~\eqref{eq:log.momentum.L.infty}, estimate~\eqref{estima-W} with $|\beta|=0$,  the bound for the mass~\eqref{mass-estimate}, and hypotheses~\eqref{hypotheses.H},
$$
\begin{array}{l}
|\mathcal A(x,t)| =\frac12\,t\log\log t\, u(x,t)\le C_1 \log\log t/(2\log t),\\[10pt]
|\mathcal B(x,t)|\le t\log \tau(t)\,\textrm{e}^{-(t-\tau(t))}\|u_0\|_{L^\infty(\mathbb{R}^2)},\\[10pt]
|\mathcal C(x,t)|\le C t\log (\tau(t))(t-\tau(t))^{-3 /2}M\big(\tau(t)\big)\le C t(t-\tau(t))^{-3/2},\\[10pt]
\mathcal D(x,t)=0\quad\mbox{if }|x|^2\ge \delta^2 t,\quad t\ge \mathcal{R}^2/\delta^2,
\end{array}
$$
which  implies the uniform convergence to 0 of all these terms as $t\to\infty$.
In order to check that also $\mathcal E$ converges to 0, we observe that $u(z,s)\le \frac C{s(\log s)^2}$ for $|z|\le \mathcal{R}+d$; see~\eqref{hypotheses.H} and~Corollary~\ref{prop:improved size estimate}.  Therefore, $(J*u(\cdot,s))\le \frac C{s(\log s)^2}$ in $\H$. Now, by using that $W(x,t)\le C t/|x|^4$, see~\eqref{eq:pointwise.estimate.W},  we get, for $|x|^2\ge\delta^2 t$ with $t$  large enough so that $|x-y|\ge \frac12|x|$ for every $y\in\H$,
$$
|\mathcal E(x,t)|\le t\log \tau(t)\frac C{|x|^4}\int_{\tau(t)}^t \frac{t-s}{s\log^2 s}\,ds\le \frac {Ct\log \tau(t)(t-\tau(t))^2}{\tau(t)(\log \tau(t))^2|x|^4}
\le  \frac{C_\delta t}{\tau(t)\log \tau(t)}\to0\quad\mbox{as }t\to\infty.
$$
\end{proof}

In view of Lemma~\ref{eq:far-field.behavior.u.v}, Theorem~\ref{thm:outer} will follow from the large time behavior for $v(\cdot,t)$ given next. Notice that the result is not immediate, since the initial data for $v$ moves with $t$.
\begin{lema}
Under the assumptions of Theorem~\ref{thm:outer}, the solution $v$ to~\eqref{eq:def.v} satisfies
\[
\lim_{t\to\infty} t\sup\{|v(x,t)-2M^*_\phi\Gamma_\a(x,t)|: x\in\R^2\}=0.
\]
\end{lema}
\begin{proof}
We have
$$
\begin{array}{rcl}
t\left|v(x,t)-2M^*_\phi\Gamma_\a(x,t)\right|&=& \underbrace{t\left|v(x,t)-\log \tau(t)M\big(\tau(t)\big)
\Gamma_\a(x,t-\tau(t))\right|}_{\mathcal{A}(x,t)}\\[10pt]
&&+\underbrace{t\log \tau(t)M\big(\tau(t)\big)\left|\Gamma_\a(x,t-\tau(t))-\Gamma_\a(x,t)\right|}_{\mathcal{B}(x,t)}\\[10pt]
&&+\underbrace{t\left|\Gamma_\a(x,t)\right|\,\left|\log \tau(t)M\big(\tau(t)\big)-2M^*_\phi\right|}_{\mathcal{C}(x,t)}.
\end{array}
$$
Notice that $t\Gamma_\a(x,t)\in L^\infty(\R^2)$, and
$$
\lim_{t\to\infty}t\sup\{|\Gamma_\a(x,t)-\Gamma_\a(x,t-\tau(t))|: x\in\R^2\}=0.
$$
Hence,
Propositions~\ref{prop-mass-decay} and~\ref{final-mass} imply, respectively, that $\mathcal{B}$ and $\mathcal{C}$ converge, uniformly in $x$, to 0 as $t\to\infty$.  Therefore, since $t-\tau(t)\ge t/2$ if $t$ is large enough, the result  will follow if we prove that
$$
\lim_{t\to\infty}(t-\tau(t)) \sup\left\{\left|v(x,t)-\log \tau(t)M\big(\tau(t)\big)
\Gamma_\a(x,t-\tau(t))\right|: x\in\mathbb{R}^2\right\}=0.
$$
This is what we do next.

Since $v$ is a solution to~\eqref{eq:def.v}, its Fourier transform satisfies
$$
\hat v(\xi,s)=\hat v(\xi, \tau(t))\textrm{e}^{-4\pi^2q(s-\tau(t))|\xi|^2}.
$$
Moreover, $\hat v(0,\tau(t))=\log \tau(t)M\big(\tau(t)\big)$. Therefore, performing the change of variables $\eta =(s-\tau(t))^{1/2} \xi$, we have, for some function $R(t)$ that will be conveniently chosen later,
$$
\begin{array}{l}
(s-\tau(t)) \Big|v(x,s)-\log \tau(t)M\big(\tau(t)\big)
\Gamma_\a(x,s-\tau(t))\Big|\\[10pt]
\displaystyle \qquad\qquad \le (s-\tau(t)) \int_{\R^2}\Big|(\hat v(\xi,\tau(t))-\hat v(0,\tau(t)))\textrm{e}^{-4\pi^2q(s-\tau(t))|\xi|^2}\Big|\,d\xi\\[10pt]
\displaystyle \qquad\qquad \le \sup_{|\xi|\le R(t)}|\hat v(\xi,\tau(t))-\hat v(0,\tau(t))|\int_{\R^N}\textrm{e}^{-4\pi^2q|\eta|^2}\,d\eta\\[10pt]
\displaystyle \qquad\qquad \quad
+2\log \tau(t) M\big(\tau(t)\big)\int_{|\eta|\ge(s-\tau(t))^{1/2} R(t)}\textrm{e}^{-4\pi^2q|\eta|^2}\,d\eta.
\end{array}
$$
On the other hand, using the Mean Value Theorem and~\eqref{first-momentum}, we obtain that
$$
\displaystyle\sup_{|\xi|\le R(t)}|\hat v(\xi,\tau(t))-\hat v(0,\tau(t))|\le \displaystyle2R(t)\log\tau(t)\int_{\R^2} |x| |u(x,\tau(t))|\, dx\le C R(t) (\tau(t))^{1/2}.
$$
Thus,  taking  $s=t$, and considering Proposition~\ref{prop-mass-decay}, we arrive at
$$
\begin{array}{l}
\displaystyle(t-\tau(t)) \left|v(x,t)-\log \tau(t)M\big(\tau(t)\big)
\Gamma_\a(x,t-\tau(t))\right|\\[10pt]
\displaystyle\qquad\qquad \le C R(t)(\tau(t))^{1/2}+ C\int_{|\eta|\ge(t-\tau(t))^{1/2}R(t)}\textrm{e}^{-4\pi^2q|\eta|^2}\,d\eta\to0 \quad\text{as } t\to\infty,
\end{array}
$$
if $R(t)(\tau(t))^{1/2}\to0$ and $(t-\tau(t))^{1/2} R(t)\to\infty$ as $t\to\infty$. These two conditions are fulfilled if,
for example,
$$
R(t)=t^{-1/2}(\log t)^\alpha, \quad 0<\alpha<1/4.
$$
\end{proof}

%%%%%%%%%%%%%%%%%%%%%%%%%%%%%%%%%%%%%%%%%%%%%%%%%%%%%%%%%
\section{Inner limit}\label{section:inner limit}
\label{sect:inner limit}
\setcounter{equation}{0}

The goal of this section is to complete the proof of
Theorem~\ref{thm:main}.
In view of Theorem~\ref{thm:outer}, what is left  is to show that the limit~\eqref{eq:main.result} is
valid uniformly in an inner set of the form $\{|x|^2<\delta t\}$ for some
$\delta>0$.  Since $\Big|\phi(x)/\log|x|\Big|\le C$ in $\mathbb{R}^2\setminus\H$, this will follow from the next result, if we take into account~\eqref{estima-W} with $|\beta|=0$.
\begin{teo}
\label{thm:inner.W} Under the assumptions of Theorem~\ref{thm:main}, for every $0<a<\a$ we have
\begin{equation}
\label{eq:inner.W}
\lim_{t\to\infty}t(\log t)^2\sup\left\{\frac{1}{\log|x|}\left|
u(x,t)-4M^*_\phi\frac{\phi(x)W(x,t)}{(\log t)^2}\right|: x\in\R^2\setminus\H,\ |x|^2\le 2at \right\}=0.
\end{equation}
\end{teo}
The advantage of this formulation in terms of $W$ is that  it is more straightforward to apply the nonlocal operator $L$ to $W$ than to $\Gamma_\a$.

In order to prove~\eqref{eq:inner.W} we perform a comparison argument in $\{|x|^2\le 2at\}$ with suitable barriers~$\omega_\ep^\pm$ which are fast enough $\ep$-close to the asymptotic limit as $t\to\infty$,
\begin{equation}
\label{eq:def.omega}
\omega_\ep^\pm(x,t)=\underbrace{4M^*_\phi\frac{\phi(x)W(x,t)}{(\log t)^2}}_{v(x,t)}\pm R_\ep(x,t), \quad R_\ep\ge0,\quad \lim_{t\to\infty}t(\log t)^2\sup_{x\in\R^2}\left|\frac{R_\ep(x,t)}{\log|x|}\right|\le \ep.
\end{equation}

We start by estimating how far is $v$ from being a solution, since this will be needed to prove that $\omega^+_\ep$ and $\omega^-_\ep$ are respectively a super- and a sub-solution.
\begin{lema} There exists a constant $C>0$ such that
\[|\partial_t v-Lv|(x,t)\le \frac C{t^2(\log t)^2}\quad\mbox{if } x\in\mathbb{R}^2\setminus\H, \quad |x|^2\le 2at,\  t>1.
\]
\end{lema}
\begin{proof} A straightforward computation shows that, for $x\in\mathbb{R}^2\setminus\H$,
$$
\begin{array}{rcl}
\displaystyle\frac1{4M_\phi^*}(\partial_t v-Lv)(x,t)&=&\displaystyle
\frac{(\partial_t W-LW)(x,t)\phi(x)}{(\log t)^2}-\frac{2\phi(x)W(x,t)}{t(\log t)^3}\\[10pt]
&&+\displaystyle\frac1{(\log t)^2}\underbrace{\int_{\mathbb{R}^2} J(x-y)(W(y,t)-W(x,t))(\phi(y)-\phi(x))\,dy}_{\mathcal{A}(x,t)}.
\end{array}
$$
In order to estimate the integral term, we write it as
$$
\begin{array}{rcl}
\mathcal{A}(x,t)&=&\displaystyle\nabla\Gamma_\a(x,t)\cdot\int_{\mathbb{R}^2} J(x-y)(y-x)(\phi(y)-\phi(x))\,dy\\[10pt]
&&\displaystyle+\left(\nabla W(x,t)-\nabla\Gamma_\a(x,t)\right)\cdot\int_{\mathbb{R}^2} J(x-y)(y-x)(\phi(y)-\phi(x))\,dy\\[10pt]
&&\displaystyle+\frac12\sum_{|\beta|=2}\int_{\mathbb{R}^2} J(x-y)D_x^\beta W(\xi,t)(y-x)^\beta(\phi(y)-\phi(x))\,dy.
\end{array}
$$
Estimating the factors involving derivatives of $W$ by~\eqref{estima-W} with $|\beta|=1$ and~\eqref{decaimiento-W} with $|\beta|=2$, the gradient of $\Gamma_\a$ by $|\nabla \Gamma_\a(x,t)|\le C |x|/t^2$,
and the difference of the values of $\phi$ at $y$ and $x$ by~\eqref{eq:differences.phi}, we get that
$\left|\mathcal{A}(x,t)\right|\le C/{t^2}$ if $x\in\mathbb{R}^2\setminus\H$. Therefore, since $W$ is a solution to problem~\eqref{eq-W},  we conclude that, for  $x\in\mathbb{R}^2\setminus\H$, $|x|^2\le 2at$,
\[
\begin{aligned}
|\partial_t v-Lv|(x,t)&\le 4M^*_\phi\left(\frac{\textrm{e}^{-t}J(x)\phi(x)}{(\log t)^2}+2\frac {\phi(x)W(x,t)}{t(\log t)^3}+\frac C{t^2(\log t)^2}\right)\\
&\le C\left(\frac{\textrm{e}^{-t}}{(\log t)^2}+\frac{\big|\log|x|\big|}{t^2(\log t)^3}+\frac 1{t^2(\log t)^2}\right)\le \frac C{t^2(\log t)^2}.
\end{aligned}
\]
\end{proof}

We now turn our attention to the correction term $R_\ep$. Notice that, in addition to having the required decay property specified in~\eqref{eq:def.omega}, $R_\ep$ should be such that $\omega^+_\ep$ and $\omega^-_\ep$ are respectively a super- and a sub-solution. To this aim, we need $\partial_t R_\ep -L R_\ep$ to be strictly positive, and with a slower decay than  $|\partial_t v-Lv|$.

Guided by our experience with the non-critical cases $N\ne 2$, our first attempt is to take a function in separated variables. A good try is to pick $\nu\in(0,1)$, and take
\begin{equation*}\label{super2}
R_\ep=\frac{\ep}{C_\nu}\underbrace{\frac{w^+_\nu(x)}{t(\log t)^2}}_{w_\nu(x,t)},\qquad C_\nu=\sup_{x\in \mathbb{R}^2\setminus\H}\left|\frac{w_\nu^+(x)}{\log|x|}\right|,
\end{equation*}
where $w^+_\nu$ is the super-solution to the stationary problem~\eqref{eq:stationary} given by Lemma~\ref{stationary-super2} for this value of $\nu$. Notice that this choice leads to the right decay for $R_\ep$. But we still have to check whether $\omega^+_\ep$ is a super-solution. Hence, we have to compute $\partial_t w_\nu-Lw_\nu$.

Using~\eqref{eq:stationary.super.2}, we get, for a certain constant $\kappa$,
\begin{equation}
\label{eq:expression.first.try}
(\partial_t w_\nu-Lw_\nu)(x,t)=-\frac{w_\nu^+(x)}{t^2(\log t)^2 }(1+2(\log t)^{-1})+\frac{\kappa}{|x|^2(\log|x|)^{2-\nu}\,t\,(\log t)^2}.
\end{equation}
Therefore,  using again~\eqref{eq:stationary.super.2}, we obtain that  there is a constant $\sigma>0$ such that
$$
(\partial_t w_\nu-Lw_\nu)(x,t)\ge \frac{\left(\kappa t-\sigma|x|^2(\log|x|)^2\right)}{|x|^2(\log |x|)^{2-\nu}\,t^2\,(\log t)^2},\quad x\in\mathbb{R}^2\setminus\H,\ t\text{ large enough}.
$$
Since $|x|\log|x|\le (\delta t)^{1/2}$ if $|x|\le \frac{(\delta t)^{1/2}}{\log t}$ and $t\ge \max\{\delta,\textrm{e}\}$, taking $\delta=\kappa/(2\sigma)$ we finally obtain, for some large enough $T$,
$$
(\partial_t w_\nu-Lw_\nu)(x,t)\ge\frac{\kappa }{2\delta t^2(\log t)^{2-\nu} },\quad x\in\mathbb{R}^2\setminus\H,\
|x|^2\le\frac{\delta t}{(\log t)^2},\ t\ge T,
$$
and we are done in this region. Unfortunately, we can not even guarantee that $w_\nu$ is as super-solution in the set $\left\{\frac{\delta t}{(\log t)^2}\le |x|^2\le 2at\right\}$, hence the need to look for an alternative.

A second possibility is to take $R_\ep=KV$, where $V$ is the super-solution that was constructed,  for a fixed $\gamma>2$, in~Lemma~\ref{lema:super1}, and $K$ is a positive  constant to be fixed. This choice also has the right decay specified in~\eqref{eq:def.omega}. Moreover,  we see from~\eqref{eq:estimate.evolutionary.supersolution}, that
$$
\partial_t V-LV\ge\frac{C}{t^2(\log t)^{\gamma-1}},\quad|x|^2\ge\frac{\delta t}{(\log t)^2},\ t\ge T.
$$
Hence, if $\gamma\in(0,3)$, this alternative works in this region. However, our estimate from below for $\partial_t V-L V$ decays too fast in the set $\left\{x\in\mathbb{R}^2\setminus\H:|x|^2\le \frac{\delta t}{(\log t)^2}\right\}$, and is not able to control the contribution of $v$ in order to make $\omega^+_\ep$ a super-solution there.

What will work is a combination of the two possibilities that we have considered,
$$
R_\ep=R_{\ep,K}(x,t)=\ep w_\nu(x,t) +K V(x,t),\quad K\ge1, \ \ep>0.
$$
 Indeed, since $V$ is a super-solution, by the very same computation performed above we have
$$
\partial_t \omega^+_\ep-L\omega^+_\ep\ge\frac{\kappa }{2\delta t^2(\log t)^{2-\nu} },\quad x\in\mathbb{R}^2\setminus\H,\
|x|^2\le\frac{\delta t}{(\log t)^2},\ t\ge T.
$$
On the other hand, we see from~\eqref{eq:expression.first.try} and~\eqref{eq:stationary.super.2} that, for some large enough $T$,
\[
|\partial_t w_\nu-Lw_\nu|(x,t)\le \frac{C}{t^2(\log t)^{2-\nu}}, \quad x\in\mathbb{R}^2\setminus\H,\ \frac{\delta^2 t}{(\log t)^2}\le |x|^2\le 2at,\ t\ge T.
\]
Thus, using~\eqref{eq:estimate.evolutionary.supersolution}, we see that there is   a large enough time $T>0$ such that, for every $K\ge 1$,
\[\begin{aligned}
\partial_t\omega^+_\ep-L\omega^+_\ep&\ge \left(\frac{C\log|x|}{t^2(\log t)^\gamma }-\frac C{t^2(\log t)^2}-\frac{\ep C}{t^2(\log t)^{2-\nu}}\right),\quad x\in\mathbb{R}^2\setminus\H:|x|^2\le 2a t,\ t\ge T.
\end{aligned}
\]
Therefore, if we first choose $2<\gamma<3$, and then $0<\nu<3-\gamma$, we finally obtain that $\omega^+_\ep$ is a super-solution in  $\{x\in\mathbb{R}^2\setminus\H:|x|^2\le 2a t,\ t\ge T_\ep\}$ for some $T_\ep>0$ large enough.
Analogously, for every $\ep>0$ and $K\ge 1$, $\omega^-_\ep$ is a sub-solution in the same set.

\begin{proof}[Proof of Theorem~\ref{thm:inner.W}. ]
Let  $\ep>0$. The outer behavior, Theorem~\ref{thm:outer}, together with~\eqref{estima-W} with $|\beta|=0$, implies that there exists a time $t_\ep$ such that
\[
t(\log t)\Big|u(x,t)-2M^*_\phi\frac{W(x,t)}{\log t}\Big|<2\ep,\qquad 2at\le |x|^2\le 4at, t\ge  t_\ep.
\]
Since $\frac{\log|x|}{\log t}\to \frac12$ and $\frac{\phi(x)}{\log |x|}\to 1$ uniformly in
$\{2at\le |x|^2\le 4at\}$ as $t\to\infty$,
\[\Big|u(x,t)-4M^*_\phi\frac{\phi(x)W(x,t)}{(\log t)^2}\Big|
\le4 \ep \frac{\log |x|}{t(\log t)^2}\le C\ep 4M^*_\phi\frac{\phi(x)W(x,t)}{(\log t)^2}\quad\mbox{if }2at\le|x|^2\le 4at,\ t>\bar t_\ep.
\]
Thus,  if $\ep$ is small,
\begin{equation*}\label{boundary}
(1-C\ep)(v(x,t)-R_{\ep,K}(x,t))\le u(x,t)\le (1+C\ep)(v(x,t)+R_{\ep,K}(x,t)),\quad 2at\le|x|^2\le 4at,\ t>\bar t_\ep,
\end{equation*}
since $R_\ep=R_{\ep,K}\ge0$. On the other hand, if we take $K\ge 1$ large enough,
\[
(1-C\ep)(v-R_{\ep,K})(x,\bar t_\ep)\le u(x,\bar t_\ep)\le (1+C\ep) (v+R_{\ep,K})(x,\bar t_\ep),\quad x\in\mathbb{R}^2\setminus\H,\  |x|^2\le 4a\bar t_\ep.
\]
Hence,
by the comparison principle we get,
\[
(1-C\ep)(v-R_{\ep,K})\le u\le (1+C\ep)(v+R_{\ep,K})\quad  x\in\mathbb{R}^2\setminus\H,\  |x|^2\le 2at,\  t\ge \bar t_\ep.
\]
Therefore, using the decay properties of $R_{\ep,K}$, we conclude that
\[ \limsup_{t\to\infty}t(\log t)^2\sup\left\{\frac{1}{\log |x|}\left|
u(x,t)-v(x,t)\right|: x\in\R^2\setminus\H,\ |x|^2\le 2at\right\}\le \widetilde C\ep.
\]
Since $\ep$ is arbitrary,  we get the desired result.
\end{proof}

\end{document}